\documentclass[12pt,a4paper]{article}
\usepackage[12pt]{extsizes}
\usepackage[T2A]{fontenc}
\usepackage[cp1251]{inputenc}
\usepackage[english]{babel}
\usepackage{amssymb,amsfonts,amsmath,mathtext}
\usepackage{cite,enumerate,float,indentfirst}
\usepackage{tikz}
\usepackage{calc}
\usepackage{textcase}
\usepackage{amsthm}
\usepackage{tocloft}
\usepackage[left=3cm,right=2cm,top=2cm,bottom=2cm,bindingoffset=0cm]{geometry}
\newtheorem{defi}{Definition}

\newtheorem{thm}{Theorem}
\newtheorem{thm*}{Theorem}

\newtheorem{lemm}{Lemma}

\makeatletter
\renewcommand{\@biblabel}[1]{#1.}
\makeatother

\title{On facets of the Newton polytope \\
		for the discriminant of the polynomial system}
\author{Irina Antipova\footnote{The first author was supported by the Theoretical Physics and Mathematics Advancement Foundation “BASIS” (№ 18-1-7-60-1),  and the Krasnoyarsk Mathematical Center, funded by the Ministry of Science and Higher Education of the Russian Federation within the framework of the establishment and development of regional Centers for Mathematics Research and Education (Agreement No. 075-02-2021-1388).},
	\, Ekaterina  Kleshkova\footnote{ The second  author was supported by the Theoretical Physics and Mathematics Advancement Foundation “BASIS” (№ 18-1-7-60-2).}}
\date{}

\begin{document}
	\maketitle
	\begin{abstract}
We study normal directions to facets of the Newton polytope of the discriminant of the Laurent polynomial system via the tropical approach. We use the combinatorial construction proposed by Dickenstein, Feichtner and Sturmfels for the tropicalization of algebraic varieties admitting a parametrization by a linear map followed by a monomial map.
	\end{abstract}

\bigskip


	\section{Introduction}
	
	  Consider a system of $n$ polynomial equations of the form
\begin{equation}
\label{ex1}
P_i := \sum_{\lambda \in A^{(i)}} a_{\lambda}^{(i)} y^{\lambda} =0, \,\, i=1, \ldots, n
\end{equation}
with unknowns $y=(y_1, \ldots, y_n)\in {( \mathbb{C} \setminus 0  )}^n$, variable complex coefficients $a=(a_{\lambda}^{(i)}) $, 
where the sets  $A^{(i)} \subset \mathbb{Z}^n$ are fixed, finite and contain the zero element $\overline{0}$, 
$\lambda = (\lambda_1, \ldots, \lambda_n)$, $y^\lambda = y_1^{\lambda_1}\cdot\ldots\cdot y_n^{\lambda_n}$. 
Solution $y(a)=( y_1 (a), \ldots, y_n (a) )$ of~(\ref{ex1}) has a polyhomogeneity property, and thus
the system usually can be written in a homogenized form by means of monomial transformations $x=x(a)$ of coefficients 
(see \cite{ATs12}). To do this, it is necessary to distinguish a collection of $n$ exponents $\omega^{(i)} \in A^{(i)}$
such that the matrix $\omega = \left( \omega^{(1)}, \ldots, \omega^{(n)}  \right)$ is non-degenerated. 
As a result, we obtain a reduced system of the form
\begin{equation}
\label{ex2}
Q_{i} := y^{\omega^{(i)}} + \sum_{\lambda \in \Lambda^{(i)}} x_{\lambda}^{(i)} y^{\lambda} -1 =0, \,\, i=1, \ldots, n,
\end{equation}
where $\Lambda^{(i)} := A^{(i)} \setminus \{ \omega^{(i)}, \overline{0} \}$ and
$x=(x_\lambda^{(i)})$ are variable complex coefficients. Denote  by $\Lambda$ the disjoint union of the sets $\Lambda^{(i)}$
and by $N$  the cardinality of this union, that is the number of variable coefficients in the system~(\ref{ex2}). 
The coefficients of the system vary in the vector space ${\mathbb C}_{x}^N$ in which points  $x=(x_{\lambda})$ 
are indexed by elements $\lambda \in \Lambda$.

Denote by $\nabla^{\circ}$ the set of all coefficients $x=\left( x_\lambda^{(i)}  \right)$ such that the polynomial
mapping $Q= \left( Q_1, \ldots, Q_n \right)$ has multiple zeroes in the complex algebraic torus ${ \left( \mathbb{C} \setminus 0 \right)^n}$, 
that is zeroes for which the determinant of the Jacobi matrix  $\frac{ \partial  Q}{ \partial y }$ of the mapping $Q$ equals zero. 
{\it The discriminant set} $\nabla$ of the system~(\ref{ex2}) is defined to be the closure of the set $\nabla^{\circ}$ 
in the space of coefficients. If $\nabla$ is a hypersurface, then the polynomial $\Delta (x)$ that defines it  
is said to be {\it the discriminant}  of the system~(\ref{ex2}). The set $\nabla$ is also called 
{\it the reduced discriminant set} of the system~(\ref{ex1}). 
According to the theory of $A$--discriminants developed in the book~\cite{GKZ},  the discriminant of a system of equations
 is appropriate to call the $(A^{(1)}, \dots , A^{(n)})$--discriminant.

 The goal our research is to construct  the tropicalization of the discriminant set $\nabla$ of the system~(\ref{ex2}) 
 and to find normal  vectors to facets of the Newton polytope of the discriminant $\Delta (x)$. 
 Recall that {\it the Newton polytope} ${\mathcal N}_{\Delta}$  of the polynomial  $\Delta (x)$  is defined to be the convex hull
 of its support  in ${\mathbb R}^n$. 

The idea of computing  the tropical discriminant  is adopted from the paper \cite{DFSt07}, 
where the general combinatorial construction  of the  tropicalization is given for algebraic varieties that admit
a parametrization in the form of the product of linear forms. We use the parametrization of the discriminant set for the
system of $n$ Laurent polynomials (\ref{ex2}) proposed and comprehensively studied in~\cite{ATs12}.
 
We write the set $\Lambda$  in the form of the block matrix $\Lambda=(\Lambda^{(1)}\vert \dots \vert \Lambda^{(n)})$
whose columns are the vectors of exponents of monomials of the system. Introduce $(n \times N)$--matrices
\begin{equation*}
\Psi := \omega^{*} \Lambda, \,\,\, \tilde{\Psi} := \Psi - {\vert\omega \vert} \chi,
\end{equation*}
where $\omega^{*} $ is the adjoint matrix to the matrix $\omega$, $\chi$ is the matrix whose $i$-th row represents the 
characteristic function of the subset $\Lambda^{(i)}\subset\Lambda$ and $\vert\omega \vert$ is the determinant of the matrix $\omega$.
 In what follows, we denote the rows of the matrices $\Psi$ and $\tilde{\Psi}$ 
by $\psi_1,\dots , \psi_n$ and $\tilde{\psi}_1,\dots , \tilde{\psi}_n$  correspondingly, 
and we denote the rows of the identity matrix $E_N$ by $e_1, \ldots, e_N$. One of the main results of this paper is
\begin{thm}
\label{th}
The vectors
	\begin{equation}\label{norm}
	e_1, \ldots , e_N , -\psi_1 , \ldots , -\psi_n, \tilde{\psi_1},\ldots, \tilde{\psi_n} \in \mathbb{Z}^N 
	\end{equation}
	define the normal directions of facets of the Newton polytope of the discriminant for the system~(\ref{ex2}). 
\end{thm}

The discriminant of the system of polynomials studied in the current paper is the generalization of 
the notion of $A$--discriminant \cite{GKZ}. 
In \cite{K91},  Kapranov proved  that $A$--discriminants characterize hypersurfaces with the birational logarithmic Gauss mapping,
and  that reduced discriminant hypersurfaces admit the parametrization by monomials depending on linear forms, which is called
 the {\it Horn-Kapranov uniformization}. The tropical analogue of the  Horn-Kapranov uniformization was obtained in~\cite{DFSt07}. 

 In the present research,  we use the parametrization of the discriminant set that occurs to be the inverse of the logarithmic
 Gauss mapping in the case when the set $\nabla$ is an irreducible hypersurface that depends on all groups of the variable
coefficients of  the system (see~\cite{ATs12}). However,  the logarithmic Gauss mapping for the studied hypersurface $\nabla$
is not always birational, thus, in general, $(A^{(1)}, \dots , A^{(n)})$--discriminants are not reduced  to $A$--discriminants, and
consequently require special research.

It is important to note that Theorem~1 does not give all normal directions to facets of the Newton polytope
but only those that are represented in the parametrization of the discriminant set $\nabla$  explicitly
as vectors of coefficients in linear forms.  Properties of the parametrization are described in Section~2 of the paper.
Further, the construction of the tropical variety and the proof of  Theorem~1 are given in  Section~3. In final Section~4, we explain
how the constructed tropical variety brings out the so-called <<hidden>> facets  of the Newton polytope for the discriminant of the system.

	\section{Parametrization of the discriminant set $\nabla$}
	
We introduce two copies of the space $\mathbb{C}^N$: one copy $\mathbb{C}^N_x$ with
coordinates $x=(x_{\lambda})$, and another copy $\mathbb{C}^N_s$ with coordinates
$s=(s_\lambda)$. In both cases, coordinates of points are indexed by  elements
$\lambda \in \Lambda$. The space $\mathbb{C}^N_s$ is considered as the space of 
homogeneous coordinates for $\mathbb{CP}_{s}^{N-1}$. It is proved in~\cite{ATs12}  that
the parametrization of the discriminant set~$\nabla$ of the system~(\ref{ex2})
is determined by the multivalued algebraic mapping from the projective space $\mathbb{CP}^{N-1}_s$
into the space ${\mathbb C}_{x}^N$ of the coefficients of the system with components
\begin{equation}
\label{par}
	x_{\lambda}^{(i)} = -\frac{s_{\lambda}^{(i)}}{\langle\tilde{\varphi_i}, s\rangle}
	\prod_{k=1}^n { \left( \frac{ \langle \tilde{\varphi}_k , s \rangle }{ \langle {\varphi}_k , s \rangle } \right) }^{\varphi_{k \lambda}},\,\, \lambda\in\Lambda^{(i)}, \,\,i=1, \ldots, n,
\end{equation}
where $\varphi_k$ and $\tilde{\varphi}_k$ are rows of  matrices $\Phi:=\omega^{-1}\Lambda$
and $\tilde{\Phi}:=\Phi-\chi$ correspondingly, 
$\varphi_{k \lambda}$ is a coordinate of the row~$\varphi_k$  indexed by $\lambda \in \Lambda^{(i)} \subset \Lambda$,
the angle brackets denote the inner product. The number of branches in~(\ref{par}) 
equals to the absolute value of the determinant $|\omega|$ but some branches 
may coincide. If the discriminant set~$\nabla$ of the system~(\ref{ex2}) is an irreducible
hypersurface that depends on all groups of variables, then the mapping~(\ref{par}) 
parametrizes it with the multiplicity that is equal to the index  $\vert {\mathbb Z}^n:H\vert$
of the sublattice $H\subset {\mathbb Z}^n$ generated by columns of the matrix $(\omega \vert \Lambda)$, i.e.,
by all exponents of the monomials from the system~(\ref{ex2}). The image of the mapping~(\ref{par})
is a hypersurface if all coordinates of  vectors $\varphi_k$, $\tilde{\varphi}_k$ are non-zero,  $k=1,\dots , n.$
	
We consider the rational mapping $\mathbb{CP}^{N-1}_s\to {\mathbb C}_w^N$ that is obtained from the mapping~(\ref{par})
as a result of raising of all its coordinates to the power~$|\omega|$. It has components
\begin{equation}
	\label{par_1}
	w_{\lambda}^{(i)} = { \left( - \frac{|\omega| s_\lambda^{(i)} }{ \langle \tilde{\psi}_i , s \rangle } \right) }^{ |\omega| } \prod_{k=1}^n { \left( \frac{ \langle \tilde{\psi}_k , s \rangle }{ \langle {\psi}_k , s \rangle } \right) }^{\psi_{k \lambda}},
\end{equation}
where $\psi_{k \lambda}$ is a coordinate of the row~$\psi_k$ indexed
 by $\lambda \in \Lambda^{(i)} \subset \Lambda$. The mapping~(\ref{par_1}) defines
 the hypersurface $\tilde{\nabla}\subset {\mathbb C}_{w}^N$ whose amoeba has the same
 asymptotic directions as the amoeba of the discriminant hypersurface~$\nabla$. 
 Recall that the {\it amoeba} $\mathcal{A}_\nabla$  of the algebraic hypersurface
$$
	\nabla=\{x\in \left({\mathbb{C}}\setminus 0\right)^{N}: \Delta(x)=0\}
$$
is the image of $\nabla$ by the mapping ${\mbox{Log}}:\left({\mathbb{C}}\setminus 0\right)^{N}\to {\mathbb{R}}^N$
defined by means of the formula
$$
	{\mbox{Log}}:(x_1,\dots , x_N)\to (\log |x_1|,\ldots , \log |x_N|).
$$
	Each component of the mapping~(\ref{par_1}) is a monomial with integer exponents composed
	with  linear functions. It is convenient to associate the mapping~(\ref{par_1}) with
	the pair of two block matrices
	\begin{equation}
	\label{matr}
	U=
	\left( { -|\omega|} E_N\left\vert \Psi^T \right\vert \tilde{\Psi}^T \right)^T, \,\,\,\,
	V=\left( {|\omega|} E_N\left\vert -\Psi^T \right\vert \tilde{\Psi}^T \right),	 
	\end{equation}
	where $E_N$ is the identity matrix.  The rows of the matrix~$U$ determine linear functions, 
	and the rows of $V$ determine  exponents of  monomials in the parametrization~(\ref{par_1}).

	\section{Tropicalization of the rational variety}
	
	Further, we pay attention to the algebraic variety~$\tilde{\nabla} \subset \mathbb{C}^N$.
Let us study its tropicalization $\tau(\tilde{\nabla})$.  According to need, we recall 
notions and facts from the tropical geometry. Elements of this theory and numerous
references to fundamental research can be found in the book~\cite{MSt15}. Since the variety~ $\tilde{\nabla}$
admits the parametrization~(\ref{par_1}) in the form of the product of linear forms, we 
implement  the general combinatorial construction 
proposed in ~\cite{DFSt07} in order to compute the tropical variety~$\tau(\tilde{\nabla})$.  
	
A tropical variety has the structure of a polyhedral fan. In particular, the
tropicali\-zation of a linear subspace $X\subset{\mathbb C}^k$ is the Bergman fan
of the matroid $M(X)$ associated with this subspace (see  \cite{A18}, \cite{FSt05}). 
The construction of the Bergman fan of a variety~$X$ is related to the notion
of the logarithmic limit-set of the variety~$X$ introduced in~\cite{B71}. 
The Bergman fan of an irreducible subvariety~$X\subset {\mathbb{C}}^k$ is 
the finite union of convex polyhedral cones with the vertex at the origin of the dimension that coincides with the 
dimension of the variety~\cite{BG84}.
	
Let us dwell in more detail on the concept of a matroid, which covers the general combinatorial 
essence of the concepts of independence in linear algebra, graph theory, etc.
 There exist several different equivalent systems of axioms that define matroid 
 (see, for instance, \cite{A18}, \cite{MSt15}, \cite{Ox11} ). Let us determine
 matroid in terms of independent sets.
	
	\begin{defi} 
		The matroid~$M$ is defined to be a pair $( \mathcal{E}, \mathcal{I} )$ that consists of a finite set~$\mathcal{E}$
		and a collection~$\mathcal{I}$ of subsets of~$\mathcal{E}$ such that
		\begin{enumerate}
			\item[(I-1)] $\emptyset \in \mathcal{I}$.
			\item [(I-2)] If $J \in \mathcal{I}$ and $I \subseteq J$, then $I \in \mathcal{I}$.
			\item[(I-3)] If $I, J \in \mathcal{I}$ and $|I|<|J|$, then there exists $j \in J-I$ such that $I \cup j \in \mathcal{I}$.
		\end{enumerate} 
		Elements of~$\mathcal{I}$ are called independent sets.
	\end{defi}
	
	Is is supposed that all one-element subsets of~$\mathcal{E}$  are independent. 
	Maximal independent sets are called {\it bases} of a matroid. From the axiom~(I-3),
	it follows that all bases of a~matroid~$M$ have the same cardinality~$r=r(M)$ that
	is called the {\it rank} of the matroid. If~$A$ is a subset of the set~$\mathcal{E}$,
	then the cardinality of maximal independent set in $A$ is called the rank of~$A$,
	and it is denoted by $r(A)$. The minimal by inclusion dependent set $C\subset\mathcal{E}$
	is said to be the  {\it cycle} of the matroid $M$. A subset $F \subseteq \mathcal{E}$ is said to be the
	{\it flat} if $r( F \cup e )> r(F)$ for all $e \notin F$. We say ~$F$ 
	is {\it proper} if its rank equals neither $0$ nor $r$. Every flat~$F$ of 
	the matroid~$M$ is represented by the incidence vector
	\begin{equation*}
	e_F := \sum_{i \in F} e_i.
	\end{equation*}	
	The vector $e_F$ is considered as an element of the {\it tropical projective space}
	$\mathbb{R}^{|\mathcal{E}|} / \mathbb{R} \mathbf{1}$, where $\mathbf{1}=(1, 1, \ldots, 1)$.
	The partially ordered set of all flats forms the lattice of flats of the matroid.
	
	The Bergman fan is one of geometric models of matroids. We use the first statement of  Theorem 2
	as the definition of the Bergman fan of the matroid~$M$.
	
	\begin{thm}[Ardila--Klivans\cite{A18}] \ 
		\begin{enumerate}
			\item[(1)] The Bergman fan~$\mathfrak{B}_M$ of a matroid~$M$ on  $\mathcal{E}$
			is the polyhedral complex in $\mathbb{R}^{|{\mathcal E}|}/ \mathbb{R} \mathbf{1}$
			that consists of the cones
			\begin{equation*}
			\sigma_{\mathcal{F}} = \textit{cone}{ \lbrace e_F: F \in \mathcal{F} \rbrace }
			\end{equation*}
			for each flag $\mathcal{F}=\{ F_1 \subsetneq \ldots \subsetneq F_l  \}$ of proper flats of $M$.
			Here, $e_F:=e_{f_1}+ \ldots + e_{f_k}$ for $F=\{ f_1, \ldots, f_k \}$. 
			\item[(2)] The tropicalization of a linear subspace~$X$ is the Bergman fan~$\mathfrak{B}_{M(X)}$
			of the mat\-roid~$M(X)$ related to the subspace~$X$.	
		\end{enumerate}
	\end{thm}

\begin{proof}[Proof of Theorem~1]
We first find the tropicali\-zation~$\tau(\tilde{\nabla})$ of 
the hypersurface~$\tilde{\nabla}$. Note that the tropicali\-zation of 
an irreducible hypersurface is the union of all cones of co\-dimension~$1$ 
of the normal fan for the Newton polytope of the polynomial defining the hypersurface.
Moreover, the collection of one-dimensional generators of the normal fan 
for the Newton polytope of the polynomial defining~$\tilde{\nabla}$
coincides with the same collection for the polynomial defining the discriminant
hypersurface~$\nabla$. 

Let $X$ be a linear subspace   in $\mathbb{C}^{N+2n}$, and let it  be the image 
of the linear mapping $s \rightarrow Us$ defined by the matrix $U$,
where $s\in {\mathbb{CP}^{N-1}_s}$. Let us consider the matroid~$M(X)$
on the set of rows of the matrix~$U$ (on the set~${\mathcal E}=\{ 1, 2, \ldots, {N+2n} \}$).
By Theorem~2, the Bergman fan $\mathfrak{B}_{M(X)}$ of the matroid~$M(X)$ 
is the tropicalization of the subspace $X$. The following lemma holds.
	
\begin{lemm}
\label{lemm_1}
The tropical variety~$\tau (\tilde{\nabla})$ is the image of the
Bergman fan~$\mathfrak{B}_{M(X)}$ under the mapping $\mathbb{R}^{N+2n} \rightarrow \mathbb{R}^{N} $
given by the matrix $V$.
\end{lemm}
	
\begin{proof}[Proof of Lemma~1]

Coordinates of the parametrization of the variety $\tilde{\nabla}$ 
given by~(\ref{par_1}) are monomials composed with linear forms. 
The linear forms are defined by the rows of the matrix~$U$
that is associated with the matroid $M(X)$ of the rank~$N$ on the set~$\mathcal{E}$.
Exponents of the monomials in parametrization are rows of the matrix~$V$.
According to \cite[Theorem~3.1]{DFSt07}, the tropicalization 
of the variety~$\tilde{\nabla}$ can be defined in terms of the matrices $U$ and $V$.
More precisely, the tropicalization is the image of the Bergman fan~$\mathfrak{B}_{M(X)}$
under the linear mapping~$V$.
\end{proof}

It follows from Lemma~1 that one-dimensional cones (rays) of the tropical 
variety $\tau (\tilde{\nabla})$ are generated by the columns of the matrix~$V$.
They do determine the normal directions for the facets of the Newton
polytope for the discriminant of the system. Thus Theorem~1 is proved.
\end{proof}
	
It is important to note that the acquired information about the structure of the normal
fan of the Newton polytope for the discriminant, as well as the parametrization~(\ref{par}),
can be used for the study of truncations  of the discriminant~$\Delta (x)$ of the system~(\ref{ex2}).
The {\it truncation} of the polynomial $\Delta (x)$ with respect to a face $h$
of the Newton polytope ${\mathcal{N}}_{\Delta}$ is defined to be the sum of all monomials
of $\Delta(x)$ whose exponents lie in the face $h$.
The geometric proof of  factorization identities for truncations 
of the classical discriminant with respect to hyperfacets of its Newton polytope
was given in the recent paper~ \cite{MSTs20}. The proof is based on the blow-up 
property of the logarithmic Gauss map for the zero set of the discriminant.
These identities were proved earlier in the book~\cite[Chapter~10]{GKZ} 
using the complicated machinery of the theory of $A$--determinants.
	\section{<<Hidden>> normal directions}
	Consider the system of two equations with two unknowns $y_1, y_2$ and three
variable coefficients~$x_1, x_2, x_3$:
\begin{equation}
\label{system_ex}
\left\{
		\begin{aligned}
		&{y_1}^2+x_1{y_1} {y_2}+x_2{y_1}{y_2}^2-1=0,\\
		&{y_2}^2+x_3{y_1}^2 y_2-1=0.\\
		\end{aligned}
\right.\end{equation}
The matrix of exponents of the system (\ref{system_ex}) has the form
\begin{equation*}
\left( \omega \vert \Lambda \right)=
	\left(
		\begin{array}{cc|ccc}
		2 & 0 & 1 & 1 & 2\\
		0 & 2 & 1 & 2 & 1 
		\end{array}
	\right).
\end{equation*}
The index of the sublattice in $\mathbb{Z}^2$ generated by its columns equals~$1$.
It is the greatest common divisor of all second order minors of the matrix.
Moreover, the matrices
\begin{equation*}
	\Phi=
	\left(
	\begin{array}{ccc}
	1/2 & 1/2 & 1\\
	1/2 & 1 & 1/2	
	\end{array}
	\right),
	\tilde{\Phi} =
	\left(
	\begin{array}{ccc}
	-1/2 & -1/2 & 1\\
	1/2 & 1 & -1/2	
	\end{array}
	\right)
\end{equation*}
do not contain zero elements, so the discriminant set~$\nabla$ of 
the system~(\ref{system_ex}) is the hypersurface, which can be parametrized 
with the multiplicity~$1$  by the mapping~$\mathbb{CP}_s^2 \to \mathbb{C}_x^3$ 
of the form
\begin{equation}
	\label{par_ex}
	\footnotesize{\left.
		\begin{aligned}				
	&x_1=- \frac{2s_1}{- s_1 - s_2 + 2s_3 } 
		{ \left( \frac{-s_1 - s_2 + 2s_3}{s_1 + s_2 + 2s_3 } \right) }^{ \frac{1}{2} } 
		{ \left( \frac{s_1 + 2s_2 - s_3}{s_1 + 2s_2 + s_3 } \right) }^{ \frac{1}{2} },\\
	&x_2=- \frac{2s_2}{- s_1 - s_2 + 2s_3 } 
		{ \left( \frac{-s_1 - s_2 + 2s_3}{s_1 + s_2 + 2s_3 } \right) }^{ \frac{1}{2} } 
		{ \left( \frac{s_1 + 2s_2 - s_3}{s_1 + 2s_2 + s_3 } \right) },\\
	&x_3=- \frac{2s_3}{s_1 + 2s_2 - s_3 } 
		{ \left( \frac{-s_1 - s_2 + 2s_3}{s_1 + s_2 + 2s_3 } \right) } 
		{ \left( \frac{s_1 + 2s_2 - s_3}{s_1 + 2s_2 + s_3 } \right) }^{ \frac{1}{2} }.
		\end{aligned}
	\right. }
\end{equation}
Here $s=(s_1, s_2, s_3)\in \mathbb{C}^3$ are the homogeneous coordinates in $\mathbb{CP}_s^2$.
	
Let us study the tropicalization~$\tau (\tilde{\nabla})$ of the rational variety~$\tilde{\nabla}\subset{\mathbb C}_w^3$
given as follows
{\	
\begin{equation}
	\label{par_ex_1}
\footnotesize{	\left.
		\begin{aligned}				
	&{w}_1={ \left( -4s_1 \right)}^{4} 
	    { \left( 2s_1 +2s_2 + 4s_3 \right) }^{-2} 
		{ \left( 2s_1+4s_2 + 2s_3 \right) }^{-2}
		{ \left( -2s_1 - 2s_2 + 4s_3 \right) }^{-2} 
		{ \left( 2s_1+4s_2 - 2s_3 \right) }^{2}  ,\\
	&{w}_2={ \left( -4s_2 \right)}^{4} 
	{ \left( 2s_1 +2s_2 + 4s_3 \right) }^{-2} 
		{ \left( 2s_1+4s_2 + 2s_3 \right) }^{-4}
		{ \left( -2s_1 - 2s_2 + 4s_3 \right) }^{-2} 
		{ \left( 2s_1+4s_2 - 2s_3 \right) }^{4}	,\\
	&{w}_3={ \left( -4s_3 \right)}^{4} 
	{ \left( 2s_1+2s_2 + 4s_3 \right) }^{-4}
	{ \left( 2s_1+4s_2 + 2s_3 \right) }^{-2}
	{ \left( -2s_1 -2s_2 + 4s_3 \right) }^{4} 
	{ \left( 2s_1 + 4s_2 - 2s_3 \right) }^{-2}.
		\end{aligned}
	\right. }
\end{equation}
}

As we noted above, the mapping ~(\ref{par_ex_1}) is associated with the matrices
\begin{equation*}
	U= \begin{pmatrix} -4 & 0 & 0 \\ 0 & -4 & 0 \\ 0 & 0 & -4 \\ 2 & 2 & 4 \\ 2 & 4 & 2 \\ -2 & -2 & 4 \\ 2 & 4 & -2   \end{pmatrix}\,\,
	\text{and}\,\,
	V= \begin{pmatrix} 4 & 0 & 0 & -2 & -2 & -2 & 2 \\ 0 & 4 & 0 & -2 & -4 & -2 & 4 \\ 0 & 0 & 4 & -4 & -2 & 4 & -2  \end{pmatrix},
\end{equation*}
that play a crucial role in constructing of the tropical 
variety~$\tau (\tilde{\nabla}) \subset \mathbb{R}^3$. 
Consider the matroid~$M$ on the set ${\mathcal E}=\{ 1, 2, 3, 4, 5, 6, 7 \}$ 
of rows of the matrix~$U$. The tropical linear space that is associated with the matroid~$M$
is the Bergman fan~$ \mathfrak{B}_M$. It is a two-dimensional fan in~$\mathbb{R}^7 / \mathbb{R}\mathbf{1}$
or a graph  depicted in Fig.~\ref{ris:pic_graf}.
		
\begin{figure}[H]
		\begin{center}
		\includegraphics[width=0.4\textwidth]{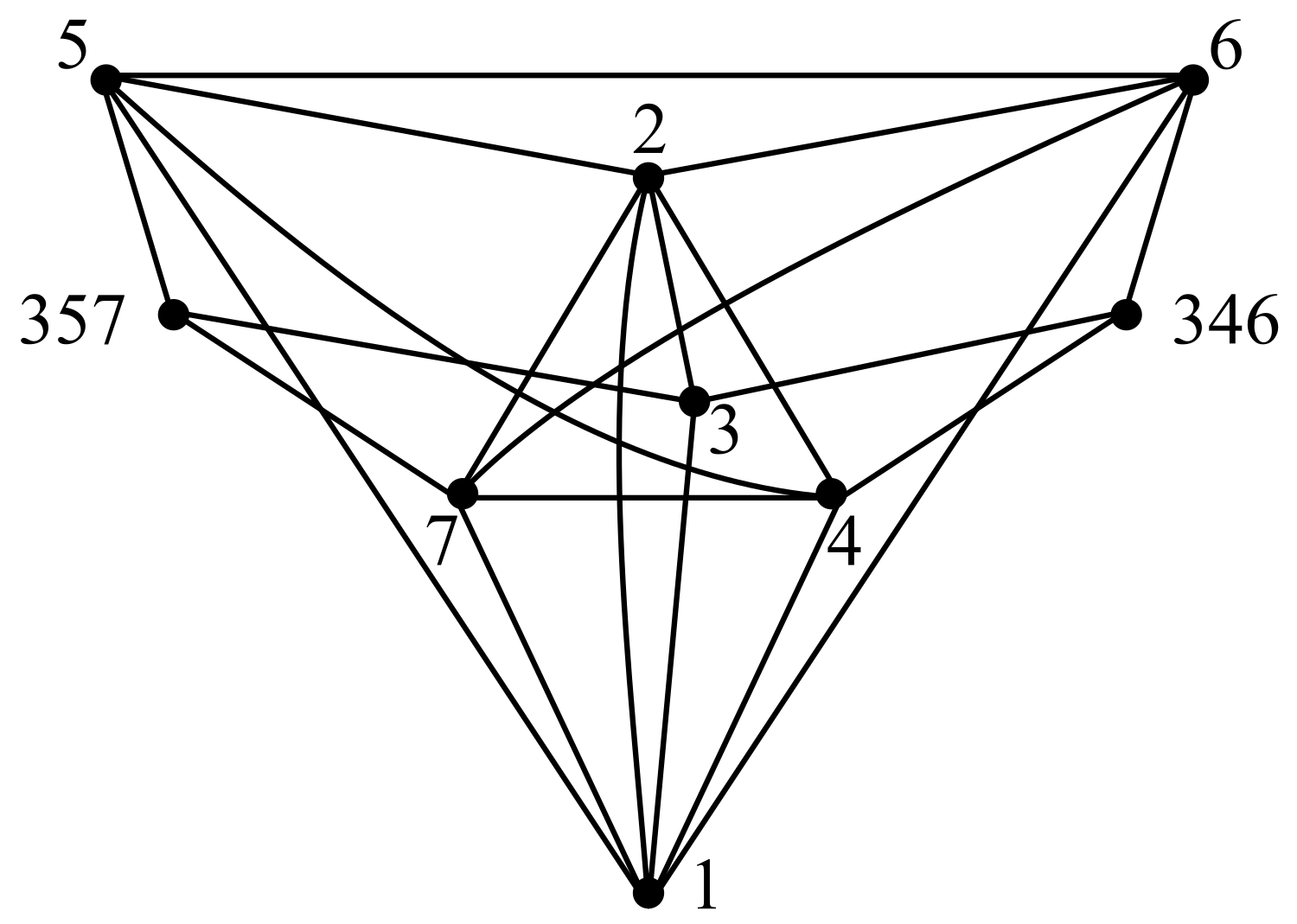}
		\caption{\small{The Bergman complex $ \mathfrak{B}_M$}}
		\label{ris:pic_graf}
		\end{center}
\end{figure}
	
The graph has nine vertices: seven of them correspond to one-element flats $1$,$2,$
$3,$ $4,$ $5,$ $6,$ $7$, and two vertices correspond to the cycles $346$ and $357$ 
of the matroid~$M$.
	
The image of the fan~$\mathfrak{B}_M$ under the mapping~$V$ is the two-dimensional
fan in~$\mathbb{R}^3$ (see Fig.~\ref{ris:vB}). It consists of sixteen cones:
fifteen of them are generated by the pairs $12$, $13$, $14$, $15$, $16$, $17$, $23$,
 $24$, $25$, $26$, $27$, $45$, $47$, $56$, $67$ of the columns of the matrix~$V$,
 and one cone is generated by the triple~$346$.
	
\begin{figure}[H]
	\begin{center}
			\begin{minipage}[h]{0.45\linewidth}
				\begin{center}	
					\includegraphics[width=0.86\textwidth]{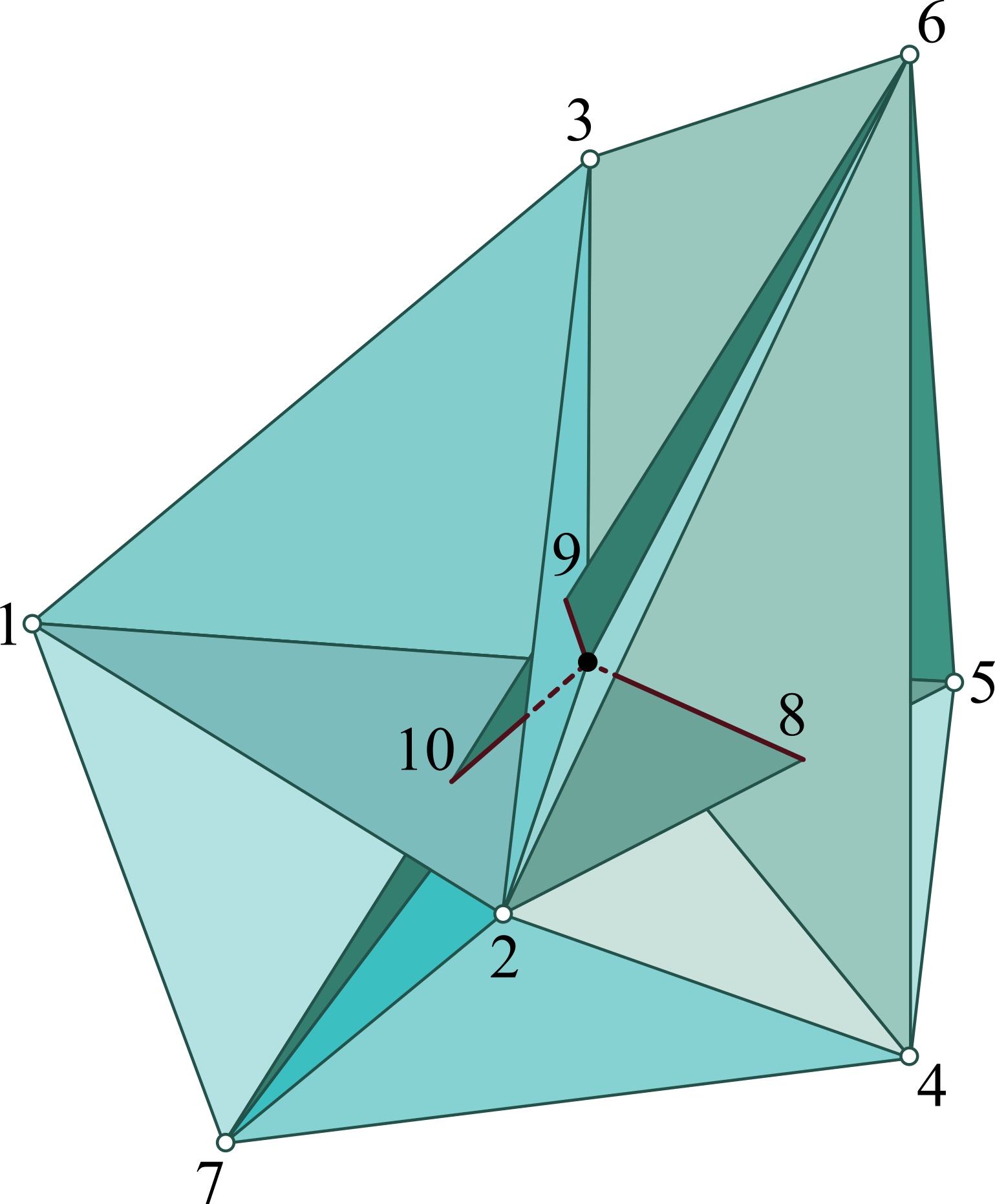}
				\end{center}
			\end{minipage}
			\hfill
			\begin{minipage}[h]{0.45\linewidth}
				\begin{center}
					\includegraphics[width=0.97\textwidth]{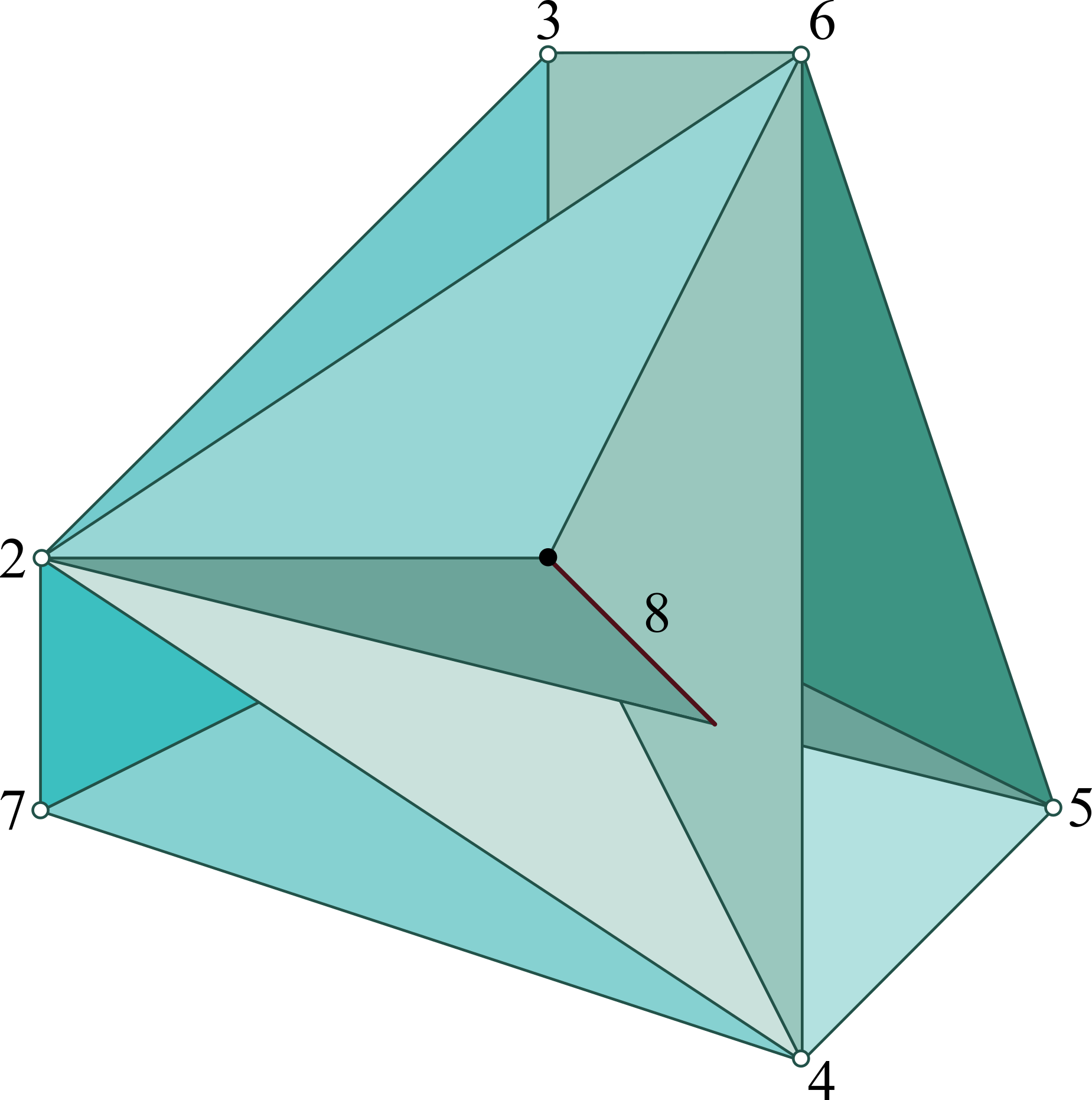}
				\end{center}
			\end{minipage}
			\caption{\small{The tropical variety~$\tau(\tilde{\nabla})$}}
			\label{ris:vB}
	\end{center}	
\end{figure}
	
The tropical variety~$\tau ( \tilde{\nabla} )$  constructed in this way allows to find the normal
fan for the Newton polytope $\mathcal{N}_{\Delta}$ of the discriminant of the system~(\ref{system_ex}). 
Seven rays of the fan $\tau (\tilde{\nabla})$ generated by the columns of the matrix~$V$ define
normal directions of facets of the Newton polytope~$\mathcal{N}_{\Delta}$. They are denoted by 
numbers $1$, $2$, $3$, $4$, $5$, $6$, $7$ in Fig.~2. 
These are exactly the normal directions obtained in Theorem~\ref{th}.
Besides these normals, there are three more rays arising as a result of  the intersection 
of cones of the fan~$\tau ( \tilde{\nabla} )$. Namely, the cones $346 $ and $ 25 $ intersect 
at the ray $\mathbb{R}_{\geqslant 0}{(-1,-1,-1)}^{T}$ (number~8 in Fig.~2),
the cones $67$ and $23$ intersect at the ray $\mathbb{R}_{\geqslant 0}{(0,1,1)}^{T}$ 
(number~9 in Fig.~2), while the cones $67$ and $12 $ intersect at the ray
$\mathbb{R}_{\geqslant 0}{(1,3,0)}^T$ (number 10 in Fig.~2). Therefore, ten 
inward  normals of the facets of the Newton polytope~$\mathcal{N}_{\Delta}$ are found:
	$$	
	\begin{array}{lcl}
	
	n_{1}=(4, 0, 0), & \qquad &  n_{6}=(-2, -2, 4),\\
	n_{2}=(0, 4, 0), & \qquad &  n_{7}=(2, 4, -2),\\
	n_{3}=(0, 0, 4), & \qquad &  n_{8}=(-1, -1, -1),\\
	n_{4}=(-2, -2, -4),  & \qquad &  n_{9}=(0, 1, 1),\\
	n_{5}=(-2, -4, -2),  & \qquad &  n_{10}=(1, 3, 0).\\
\end{array}
$$

The discriminant of the system~(\ref{system_ex}) is computed using 
the computer algebra system for polynomial computations {\sc Singular}~\cite{DGPS}.
It has the following form
	{\scriptsize{
			\begin{equation*}
			\begin{split}
			&\Delta(x)=
			432x_1^{11}x_3^3
			-1296x_1^9x_2^2x_3^3
			-64x_1^9x_3^5
			+1296x_1^7x_2^4x_3^3
			+192x_1^7x_2^2x_3^5
			-432x_1^5x_2^6x_3^3
			-192x_1^5x_2^4x_3^5{}\\
			&+64x_1^3x_2^6x_3^5
			+2976x_1^8x_2x_3^4
			-20916x_1^6x_2^3x_3^4
			-384x_1^6x_2x_3^6
			+9576x_1^4x_2^5x_3^4
			+2928x_1^4x_2^3x_3^6{}\\
			&-3300x_1^2x_2^7x_3^4
			-1248x_1^2x_2^5x_3^6
			+432x_2^7x_3^6
			+4032x_1^9x_3^3
			+13272x_1^7x_2^2x_3^3
			-3684x_1^7x_3^5{}\\
			&-45864x_1^5x_2^4x_3^3
			+8580x_1^5x_2^2x_3^5
			+432x_1^5x_3^7
			+43560x_1^3x_2^6x_3^3
			-91176x_1^3x_2^4x_3^5
			-1200x_1^3x_2^2x_3^7{}\\
			&-15000x_1x_2^8x_3^3
			+96x_1x_2^6x_3^5
			+12000x_1x_2^4x_3^7
			-4068x_1^8x_2x_3^2
			+25136x_1^6x_2^3x_3^2
			-10192x_1^6x_2x_3^4{}\\
			&-50568x_1^4x_2^5x_3^2
			+3584x_1^4x_2^3x_3^4
			-1056x_1^4x_2x_3^6
			+42000x_1^2x_2^7x_3^2
			-467344x_1^2x_2^5x_3^4
			+97335x_1^2x_2^3x_3^6{}\\
			&-12500x_2^9x_3^2
			+4800x_2^7x_3^4
			+4800x_2^5x_3^6
			-12500x_2^3x_3^8
			+11360x_1^7x_3^3
			+102792x_1^5x_2^2x_3^3
			+3552x_1^5x_3^5{}\\
			&+172032x_1^3x_2^4x_3^3
			-132850x_1^3x_2^2x_3^5
			-509960x_1x_2^6x_3^3
			+567418x_1x_2^4x_3^5
			+15000x_1x_2^2x_3^7
			{}\\
			&-36944x_1^6x_2x_3^2
			+108528x_1^4x_2^3x_3^2
			+20003x_1^4x_2x_3^4
			+53328x_1^2x_2^5x_3^2
			-856566x_1^2x_2^3x_3^4
			-3300x_1^2x_2x_3^6{}\\
			&-118000x_2^7x_3^2
			+417147x_2^5x_3^4
			-118000x_2^3x_3^6
			-64x_1^7x_3
			+7328x_1^5x_2^2x_3
			+8192x_1^5x_3^3
			-14464x_1^3x_2^4x_3{}\\
			&+477776x_1^3x_2^2x_3^3
			-64x_1^3x_3^5
			+7200x_1x_2^6x_3
			-621232x_1x_2^4x_3^3
			+130440x_1x_2^2x_3^5
			-99720x_1^4x_2x_3^2
			{}\\
			&+264032x_1^2x_2^3x_3^2
			-27048x_1^2x_2x_3^4
			-340328x_2^5x_3^2
			-340328x_2^3x_3^4
			-512x_1^5x_3
			+57952x_1^3x_2^2x_3{}\\
			&-512x_1^3x_3^3
			+56736x_1x_2^4x_3
			+320064x_1x_2^2x_3^3
			+432x_1^4x_2
			-864x_1^2x_2^3
			-61440x_1^2x_2x_3^2
			+432x_2^5{}\\
			&-265536x_2^3x_3^2
			+432x_2x_3^4
			-1024x_1^3x_3
			+125568x_1x_2^2x_3
			+3456x_1^2x_2
			+3456x_2^3
			+3456x_2x_3^2
			+6912x_2.
			\end{split}
			\end{equation*}
	}}

	\begin{figure}[H]
			\begin{center}
			\begin{minipage}[h]{0.45\linewidth}
			\begin{center}	
			\includegraphics[width=0.89\textwidth]{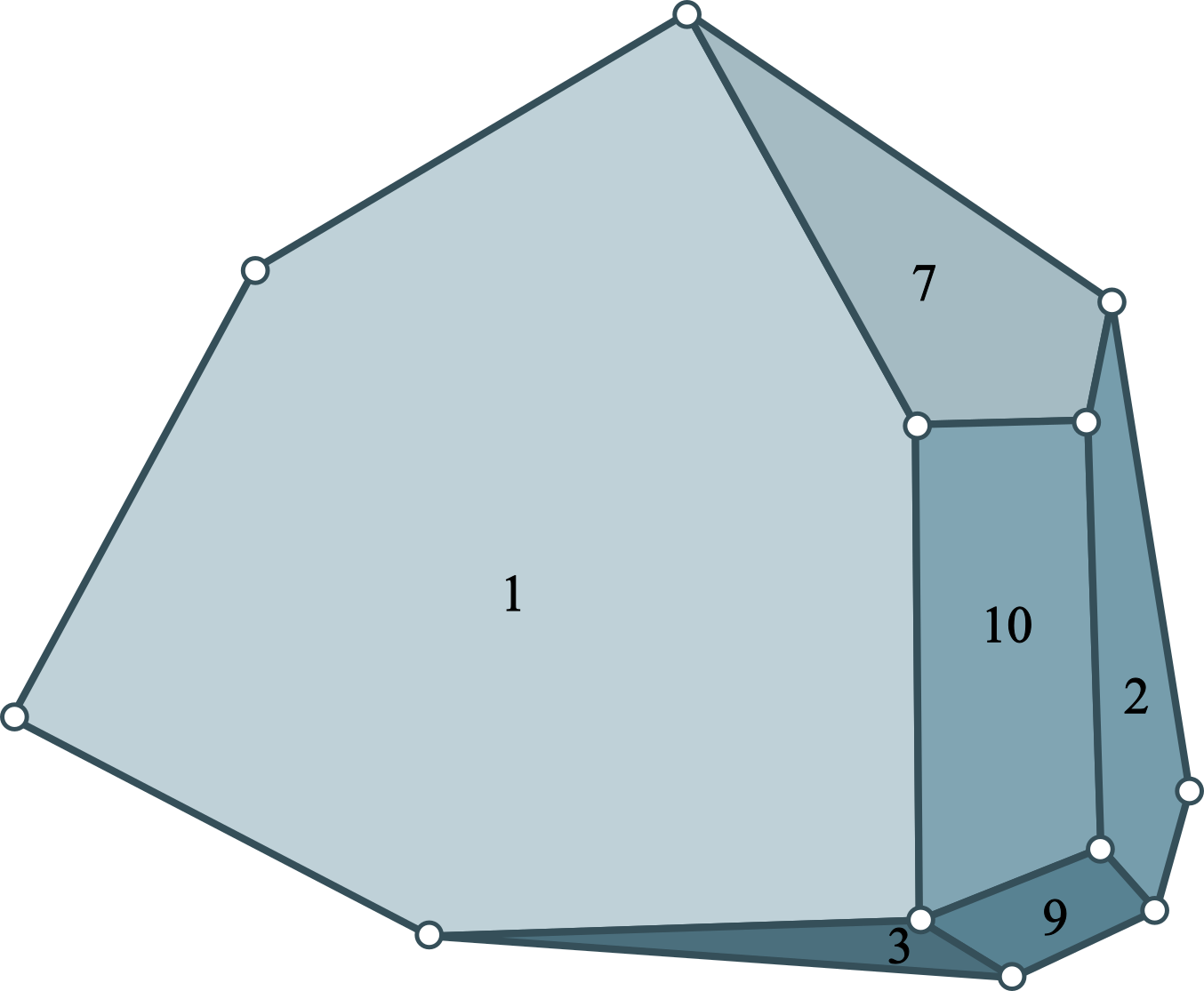}
			\end{center}
			\end{minipage}
			\hfill
				\begin{minipage}[h]{0.45\linewidth}
				\begin{center}
				\includegraphics[width=0.99\textwidth]{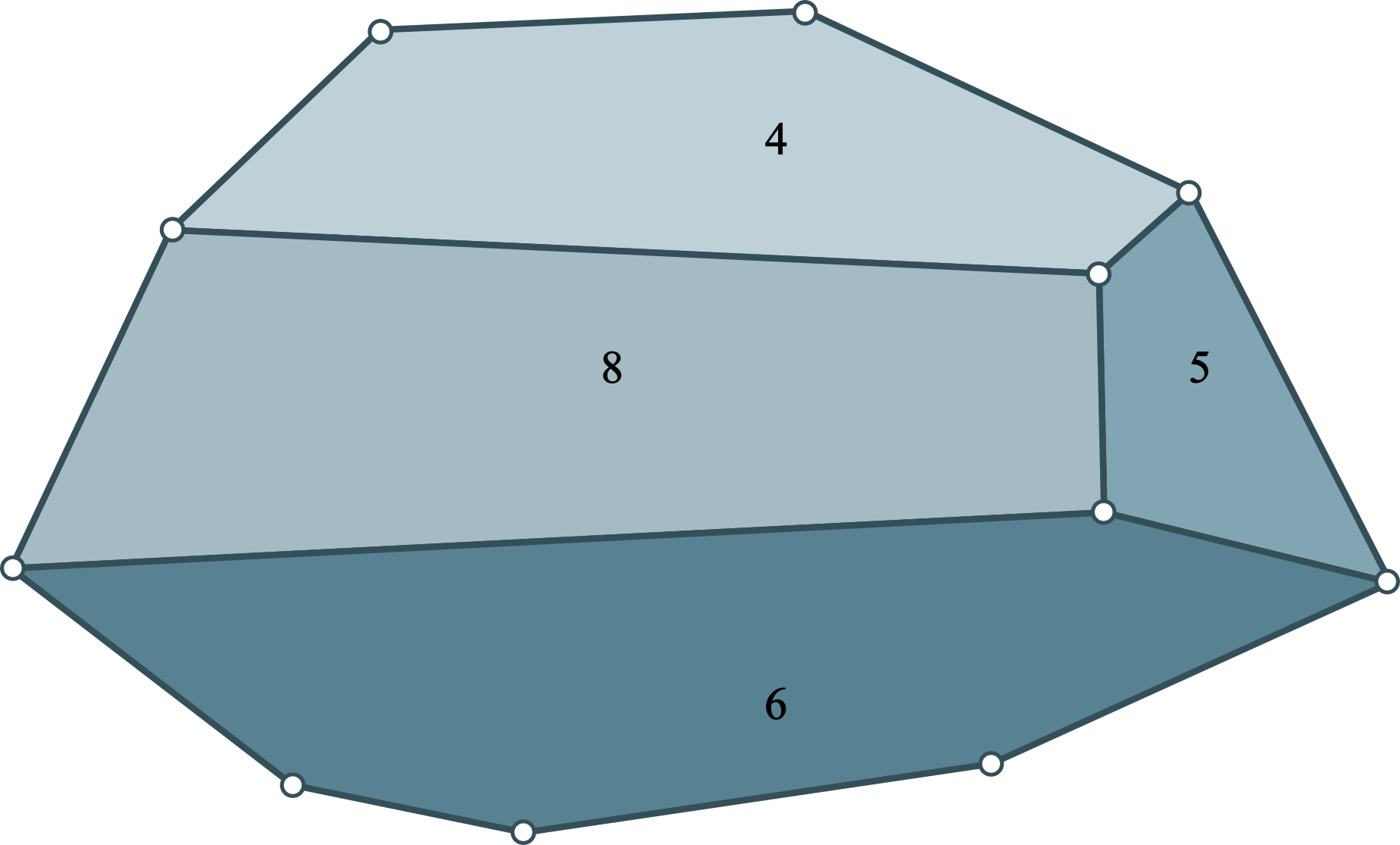}
			\end{center}
			\end{minipage}
			\caption{\small{The Newton polytope $\mathcal{N}_{\Delta}$}}
			\label{ris:mN}
		\end{center}	
		\end{figure}
Its Newton polytope~$\mathcal{N}_{\Delta}$ is depicted in Fig.~\ref{ris:mN} 
from two angles. It has ten two-dimensional faces. They are numbered in accordance with
the sequence of rays of the polyhedral fan in Fig.~2. Therefore, all
facets of the Newton polytope are detected.

%


	\addcontentsline{toc}{section}{References}
	\begin{center}
	\renewcommand{\refname} 
	\end{center}
\bigskip

\noindent {\bf Authors' addresses}: 

\noindent Irina Antipova, Siberian Federal University, 79 Svobodny pr., 660041 Krasnoyarsk, Russia, iantipova@sfu-kras.ru

\noindent Ekaterina Kleshkova, Siberian Federal University, 79 Svobodny pr., 660041 Krasnoyarsk, Russia, ekleshkova@gmail.com

\end{document}